\documentclass[6pt]{article}
\usepackage[english]{babel}

\usepackage{amsmath, amsthm}
\usepackage{amscd}
\usepackage{amsfonts}
\usepackage{amssymb}
\usepackage{url}
\usepackage[toc,page]{appendix}

\title{Compact fibrations with hyperk\"ahler fibers}
\author{Rodion N. D\'eev\footnote{The article was prepared within the framework of the Academic Fund Program at the National Research University Higher School of Economics (HSE) in 2015 (grant 15-05-0033) and supported within the framework of a subsidy granted to the HSE by the Government of the Russian Federation for the implementation of the Global Competitiveness Program.}}

\newcommand{\arr}{\xrightarrow}
\DeclareMathOperator{\Diff}{Diff}
\newcommand{\Z}{\mathbb{Z}}
\newcommand{\bC}{\mathbb C}
\newcommand{\R}{\mathbb{R}}

\newcommand{\bH}{\mathbb{H}}
\newcommand{\fX}{\mathfrak{X}}
\newcommand{\cI}{\mathcal{I}}
\newcommand{\fI}{\mathfrak{I}}

\DeclareMathOperator{\Gr}{Gr}
\DeclareMathOperator{\GL}{GL}
\DeclareMathOperator{\Sp}{Sp}
\DeclareMathOperator{\SO}{SO}
\DeclareMathOperator{\rU}{U}
\newcommand{\bP}{\mathbb{P}}
\DeclareMathOperator{\Teich}{Teich}
\DeclareMathOperator{\Per}{Per}
\DeclareMathOperator{\Hom}{Hom}
\newcommand{\per}{per}
\DeclareMathOperator{\SU}{SU}
\DeclareMathOperator{\Id}{Id}
\DeclareMathOperator{\Cau}{Cau}
\DeclareMathOperator{\Tw}{Tw}

\newtheorem{fact}{Proposition}[section]
\newtheorem{thm}{Fact}[section]
\newtheorem*{thrm}{Theorem}
\newtheorem*{lm}{Lemma}
\theoremstyle{definition}
\newtheorem{defn}{Definition}[section]

\begin{document}
\maketitle
\begin{abstract}
Essential dimension of a family of complex manifolds is the dimension of the image of its base in the Kuranishi space of the fiber. We prove that any family of hyperk\"ahler manifolds over a compact simply connected base has essential dimension not greater than $1$. A similar result about families of complex tori is also obtained.
\end{abstract}
\tableofcontents

\newpage
\section{Introduction}
Let $\bH$ be the skew-field of quaternions.

\begin{defn}A manifold $M$ with a left $\bH$-action in the tangent bundle is called {\it almost hypercomplex}. Any quaternion $q$ with $q^2=-1$ defines an almost complex structure on $M$, and if all these complex structures are integrable, $M$ is called {\it hypercomplex}.
\end{defn}

One does not need to check the integrability condition for all points $q$ in the 2-sphere $q^2=-1$, but only for three quaternions spanning the space of imaginary ones. Any hypercomplex structure admits a unique torsion-free connection preserving the hypercomplex structure, called {\it Obata connection} \cite{Ka, Ob}.

\begin{defn}A Riemannian hypercomplex manifold $(M,g)$ is called {\it hyperk\"ahler}, if its Obata and Levi-Civita connections are equal. 
\end{defn}

One can think of hyperk\"ahler manifolds as of Riemannian manifolds with three integrable almost complex structures $I$, $J$, $K$ such that $IJ = -JI = K$ which are K\"ahler with respect to the Riemannian metric, or as of Riemannian manifolds with holonomy in the group $\Sp(n)$. When $n=1$, $\Sp(1) = \SU(2)$ and all the hyperk\"ahler manifolds are either K3 surfaces or tori.

One can wonder what are the families of hyperk\"ahler manifolds, i.~e. submersions of complex manifolds with hyperk\"ahler fibers. This question is ultimately closely related to the variations of Hodge structures, because any family of complex manifolds gives rise to a variation of Hodge structures. Some complex manifolds are completely determined by their Hodge structures; statements of such type are known as ``Torelli theorems''. Although the global Torelli theorem for hyperk\"ahler manifolds was proved in~2009 \cite{V1}, Torelli theorem for K3 surfaces has much longer history: its local version was established in~1964 by G.~N.~Tjurina \cite{T}, whereas its global version was obtained in~1977 by Vik.~S.~Kulikov \cite{Ku} (earlier advances in this topic include a 1971 paper by I.~I.~Pyatetski-Shapiro and I.~R.~Shafarevich \cite{PShSh}; a 1978 paper \cite{Bo1} of Bogomolov is also worthy to mention). Because of that study of the families of complex manifolds was reduced to study of variations of Hodge structures since olden times. The first advance in this direction is due to Deligne and Griffiths.

\begin{thm}[P.~R.~Deligne, Ph.~A.~Griffiths, 1970]Any variation of Hodge structures over a compact simply connected base is trivial.
\end{thm}

Of course, Griffiths has proved much stronger statement, called ``Theorem of the Fixed Part'' \cite[Ch.~II, Application~7]{G}, but we shall not need it in its full generality. It follows from this fact that any polarized family of hyperk\"ahler manifolds over a compact simply connected base is trivial. The following definition asserts that one cannot drop out the polarization condition in the Deligne's and Griffiths' statement.

\begin{defn}If $X$ is a hypercomplex (for example, hyperk\"ahler) manifold, then any imaginary unit quaternion $q$ defines a complex structure on $X$. This gives rise to a nontrivial almost complex structure on the space $X \times \bC P^1$, where we identify $\bC P^1$ with the unit sphere in the space of imaginary quaternions: namely, tangent space at point $(x,q)$ splits as $T_xX \oplus T_q\bC P^1$, and one can put the complex structure $q$ on $T_xX$ and the standard one on $T_q\bC P^1$. 

\begin{thrm}[M.~Obata, 1953 \cite{Ob}, S.~Salamon, 1982 \cite{S}, D.~Kaledin, 1996 \cite{Ka}] This almost complex structure is integrable.
\end{thrm}

The product $X \times \bC P^1$ with this complex structure is denoted by $\Tw(X)$ and called the {\it twistor space}.
\end{defn}

The projection $\Tw(X)\to\bC P^1$ is holomorphic and defines a nontrivial family of hyperk\"ahler manifolds. It is not hard to prove that the twistor space cannot bear any K\"ahler form \cite[Proposition 2.13]{Hi,V3}. The concept of twistors has appeared in works of R.~Penrose in the late 1960ies \cite{P}, and has been developed by M.~F.~Atiyah, N.~J.~Hitchin and I.~M.~Singer \cite{AHS} (in context of 4-dimensional Riemannian geometry) and S.~Salamon \cite{S} (in context of quaternionic K\"ahler geometry).

Seeking for a way to generalize a well-known fact about isotriviality of the complete families of elliptic curves, R.~E.~Borcherds, L.~Katzarkov, T.~Pantev and N.~I.~Shepherd-Barron proved in 1997 the following theorems, which allow to drop out the condition of simply connectedness in Griffiths' statement at least for fibrations with fibers K3 surfaces (but need some extra restrictions on the fibration).

\begin{thm}[R.~E.~Borcherds, L.~Katzarkov, T.~Pantev and N.~I.~Shepherd-Barron, 1997, \cite{BKPShB}]Any complete family of minimal K\"ahler surfaces of Kodaira dimension $0$ and constant Picard number is isotrivial.
\end{thm}

\begin{thm}[R.~E.~Borcherds, L.~Katzarkov, T.~Pantev and N.~I.~Shepherd-Barron, 1997, \cite{BKPShB}]Any family  of smooth polarized K3 surfaces with polarization of degree 2 over a projective variety is isotrivial.
\end{thm}

They dealt separately with the cases of hyperelliptic, Enriques and K3 fiber, and the latter was the essential one. Their technique heavily uses the theory of automorphic forms, namely the properties of the Borcherds' automorphic form $\Phi_{12}$ on the Cartan symmetric space for the group $\mathrm{O}(\mathrm{II}_{2,26})$, and one cannot prove similar results for fibrations with hyperk\"ahler fibers of dimension greater than 28 in the same way. It seems that the only generalization of their result to arbitrary hyperk\"ahler manifolds is the following theorem of K.~Oguiso.

\begin{thm}[K.~Oguiso, 2000, \cite{Og}]Let $\fX\to\Delta$ be a nontrivial family of hyperk\"ahler manifolds over a disc. Then the set of points where the Picard number of the fiber jumps is a dense countable subset of a disc.
\end{thm}

In this paper we give alternative proofs of the Facts 1.2 and 1.4 via newly developed techniques (whereas the Fact 1.3 cannot be proved with them).

The global Torelli theorem is known for curves, tori, hyperk\"ahler manifolds and some exotic cases such as cubic threefolds \cite[Ch.~XII~and~XIII]{G}. Curves and threefolds are automatically polarized, so we are not interested in them. Recall that the {\it essential dimension} of a smooth complex fibration is the dimension of the base in the Kuranishi space of the fiber.  The main theorem we are going to prove is the following.

\begin{thrm}Let $\fX\to B$ be a smooth fibration of compact complex manifolds with smooth hyperk\"ahler fibers and simply connected base. Then its essential dimension is no greater than one.
\end{thrm}

For the fibrations by complex tori a similar result holds.

\begin{thrm}Let $\fX\to B$ be a smooth fibration of compact complex manifolds with fibers $n$-dimensional complex tori and simply connected base. Then its essential dimension is no greater than $\frac{n(n-1)}2$.
\end{thrm}

Here is a brief outline of the paper. In Section~2 we give some well-known facts about hyperk\"ahler manifolds, such as the global Torelli theorem. In Section~3 we describe some geometry of the moduli space of hyperk\"ahler manifolds, which is also widely known. Following those facts, we prove the main theorems and derive the Borcherds--Katzarkov--Pantev--Shepherd-Barron and Oguiso theorems. Subsequently, in Section 4 we give some other relevant observations on the fibrations with hyperk\"ahler fibers.

\section{Preliminaries}
It follows from a~straightforward calculation that a~form $$\Omega_I(u,v) = g(Ju,v) + \sqrt{-1}g(Ku,v)$$ on a~compact hyperk\"ahler manifold $(M,g,I,J,K)$ is~holomorphic with~respect~to the~complex structure $I$. 

Hyperk\"ahler manifolds with $h^{2,0}(M,I)=1$ are known as {\it simple}. For simple hyperk\"ahler manifolds the~cohomology class $[\Omega_I]$ spans the~line $H^{2,0}(M,I)\subset H^2(M,\bC)$. The famous Bogomolov decomposition theorem \cite{Bo2} states that any hyperk\"ahler manifold admits a finite cover which is a product of torus and simple hyperk\"ahler manifolds, so we are interested in the simple ones. From now onwards, until otherwise stated, while speaking of a hyperk\"ahler manifold $(M,g,I,J,K)$ as of complex manifold, we would consider the complex structure $I$. Consider a~complex family $(\fX,\cI)$ of~simple hyperk\"ahler manifolds over a~simply connected base $B$ (so that $R^2\pi_*\bC$ is a~trivial bundle), and let $X_b$ denote the~fiber over a~point $b\in B$. Then we can define the {\it period map} $B\arr{\per}\bP(H^2(X,\bC))$ which sends the~point $b$ to the line~spanned by the~class $[\Omega_{\cI|_{X_b}}]$.

If the base $B$ is not simply connected, the period map is defined as a map from the universal cover $\widetilde{B}$. The fundamental group of the base acts on the universal cover and on the period space (as the monodromy group of the local system $R^2\pi_*\bC$), and the period map is equivariant with respect to these two actions. One can try to obtain a period map from the base $B$ into the quotient of the period space by this action, but this quotient can have very poor topology. It is known to be an orbifold in the case of the polarized fibration with fiber K3 surface. However, in other cases the topology on the factor can be very non-Hausdorff, e.~g. every two non-empty subsets intersect.

\subsection{Bogomolov--Beauville--Fujiki form}
For a K3 surface $X$ the intersection form is a bilinear form on the space $H^2(X,\R)$; by the Hodge index theorem, its signature is $(3,19)$. We shall need a similar form on the second cohomology space of a hyperk\"ahler manifold, which would have the signature $(3,b_2-3)$.

\begin{thm}[F.~A.~Bogomolov, 1978 \cite{Bo1}, A.~Beauville, 1983 \cite{Be}, A.~Fujiki, 1985, 1987 \cite{F}]Let $X$ be a hyperk\"ahler manifold of real dimension $4n$. There exists a unique primitive quadratic form $q\colon H^2(X,\Z)\to\Z$ and a constant $c$ such that for any $\alpha\in H^2(X)$ one has
$$\int_X \lambda^{2n} = cq(\alpha)^n$$
and for non-zero $\sigma\in H^{2,0}(X)$ one has $q(\sigma+\overline{\sigma})>0$.
\end{thm}

For more details, see \cite{OG}.

This form is uniquely determined by this condition. One can write down an explicit formula for $q$ (here we use same letter $q$ for the polarization of the Bogomolov--Beauville--Fujiki quadratic form):
\begin{multline*}
cq(\alpha,\beta) = 2\int_X\alpha\wedge\beta\wedge\Omega_I^{n-1}\wedge\overline{\Omega_I}^{n-1} - \\ - \frac{n-1}n\frac{(\int_X\alpha\wedge\Omega_I^{n-1}\wedge\overline{\Omega_I}^n)(\int_X\beta\wedge\Omega_I^n\wedge\overline{\Omega_I}^{n-1})}{\int_X\Omega_I^n\wedge\overline{\Omega_I}^n},
\end{multline*}

where the positive constant on the left-hand side is needed for the form $q$ to be integer.

It is positive definite on the real part of the space spanned by $\Omega_I$, $\overline{\Omega_I}$ and the K\"ahler form $\omega$, and negative definite on the primitive forms (i.~e. it has signature $(3,b_2-3)$). The image of the period map lies in the set of lines spanned by the classes $\alpha$ such that $q(\alpha) = 0$ and $q(\alpha+\overline{\alpha})>0$ (or, equivalently, $q(\alpha,\overline{\alpha})>0$).

\subsection{Global Torelli theorem}
The main reference for this section is \cite{V1}.

\begin{defn}Let $(X,g)$ be a Riemannian manifold, and $\fI$ be the set of complex structures on $X$ which can be extended to hyperk\"ahler structures. The group $\Diff_0(X)$ of oriented diffeomorphisms of $X$ act on $\fI$. The {\it Teichm\"uller space} $\Teich(X)$ of $X$ is the factor $\fI/\Diff_0(X)$.
\end{defn}

The space $\Teich(X)$ admits a period map $(X,I)\mapsto H^{2,0}(X,I)\in\bP(H^2(X,\bC))$ into $\bP(H^2(X,\bC))$ in the same way as the base of any fibration with hyperk\"ahler fibers (provided that the base is simply connected). The image of the period map lies in the {\it period space} $\Per(X)=\{[\alpha]\in\bP(H^2(X,\bC))\mid q(\alpha)=0,q(\alpha,\overline{\alpha})>0\}$. Actually, $\Teich$ is not Hausdorff, but its non-Hausdorff points correspond to the bimeromorphically equivalent hyperk\"ahler manifolds. Moreover, there exists a Hausdorff space $\Teich_b$ with map $\Teich\to\Teich_b$ such that any map from $\Teich$ to a Hausdorff space factorizes through this map (so the period map $\Teich_b(X)\to\Per(X)$ is well-defined).

\begin{thm}[M.~Verbitsky, 2009 \cite{V1}]The map $\per\colon\Teich_b\to\Per$ is a diffeomorphism on each connected component of $\Teich_b$.
\end{thm}

\section{Geometry of the period space}
\subsection{Positive Grassmannians}

\begin{defn}Let $V$ be a vector space over $\R$ with a non-degenerate bilinear form $\langle\cdot\mid\cdot\rangle$. We shall consider a subset in the oriented Grassmannian $\Gr_2(V)$ consisting of oriented 2-planes with positive definite restriction of the bilinear form (we shall call such planes {\it positive}). This subset is called {\it the positive Grassmannian} and denoted $\Gr_{++}(V)$.
\end{defn}

\begin{lm}[see e.~g. C.~LeBrun, 1993 \cite{LB}]The set $\Gr_{++}(V)$ is in one-to-one correspondence with the projectivization of the set of vectors $v\in V\otimes\bC$ such that $\langle v\mid v\rangle = 0$ and $\langle v\mid\overline{v}\rangle>0$ (we shall call them {\normalfont positive null-vectors}).
\end{lm}
\begin{proof}If $W\subseteq V$ is a point in $\Gr_{++}(V)$, then the cone of null-vectors in $W\otimes\bC$ consists of two lines, which are interchanged by the complex conjugation. If $w$ is a vector spanning one of the lines, then one needs to be $\langle w\mid\overline{w}\rangle>0$ for the inner product on $W$ to be positive definite, and vectors $w$ and $\overline{w}$ correspond to two copies of the plane $W$ coming with different orientations: if the basis $\{w+\overline{w},i(w-\overline{w})\}$ is positively oriented, then the null-vector corresponding to the plane $W$ is $w$, and $\overline{w}$ otherwise.

Conversely, if $v\in V\otimes\bC$ is such that $\langle v\mid v\rangle=0$ and $\langle v\mid\overline{v}\rangle>0$, then vectors $v+\overline{v}$ and $i(v - \overline{v})$ are both real and linearly independent, so they span a plane in $V$ with an orientation. It is easy to see that the metric on this plane is positive definite.
\end{proof}

Let $v\in V$ be a non-zero vector. Then the subset of positive 2-planes orthogonal to $v$ is exactly $\Gr_{++}(v^\perp)\subset \Gr_{++}(V)$. Indeed, the positive 2-plane $W$ is orthogonal to $v$ if and only if both $w$ and $\overline{w}$ are.

The tangent space to $\Gr_{++}(V)$ at the point $W$ is the same as to the Grassmannian, $\Hom(W,V/W)$. However, $W$ is an oriented plane with positive definite metric, so it can be regarded as a one-dimensional complex vector space, turning the tangent space into a complex one. It gives an almost complex structure on $\Gr_{++}(V)$, which is the same as the restriction of the complex structure from the complex quadric $\{\langle v\mid v\rangle=0\}\subset\bP(V\otimes\bC)$.

\begin{fact}$\Gr_{++}(\R^{2,n})$ is a Stein manifold.
\end{fact}
\begin{proof}
$\Gr_{++}(\R^{2,n}) = \SO(2,n)/\SO(2)\times\SO(n)$, so it is a symmetric domain of non-compact type. Due to a theorem of \'E. Cartan \cite{C} it can be holomorphically embedded into a complex space as a bounded domain.

See \cite[Theorem 7.1]{He} for a complete proof.
\end{proof}

\begin{fact}$\Gr_{++}(\R^{2,1})$ is a topological disc.
\end{fact}
\begin{proof}
Let $u,v,w$ be the orthogonal basis of $\R^{2,1}$ such that $\|v\|^2=\|w\|^2=1$ and $\|u\|^2=-1$. Any 2-plane $W$ with positive definite restriction of the metric projects along $u$ onto the plane $W_0=\langle v,w\rangle$ isomorphically (just because $W$ cannot contain the kernel of this projection, the line spanned by $u$). So $W = \langle v+au, w+bu \rangle$ for some real numbers $a$ and $b$, and different pairs of numbers define different planes. The restriction of metric on $W$ is positive definite iff for any real numbers $x,y$ (at least one of which is not equal to zero) one has $0 < \|x(v+au) + y(w+bu)\|^2 = x^2 + y^2 - (ax+by)^2$, which is equivalent to the condition $a^2 + b^2 < 1$.

One can also send a positive plane in $\R^{2,1}$ into its orthogonal, which is a line spanned by a negative vector, and obtain a representation of $\Gr_{++}(\R^{2,1})$ as the projectivization of the negative cone in $\R^{2,1}$, which is precisely the Cayley--Klein model for the Bolyai--Lobachevskian plane.
\end{proof}

\begin{fact}$\Gr_{++}(\R^{n,m})$ can be retracted onto $\Gr_{++}(\R^{n,0})$. In particular, $\Gr_{++}(\R^{2,n})$ is contractible.
\end{fact}
\begin{proof}Let $l \subset V$ be a line spanned by a negative vector. Then if $W$ is a positive 2-plane, then its projection along $l$ in $l^\perp$ is also a positive 2-plane. This defines a fibration $\Gr_{++}(V)\to\Gr_{++}(l^\perp)$, and its fiber over a plane $W' \subseteq l^\perp$ is a set of positive 2-planes in the linear hull of $W'$ and $l$, i.~e. $\Gr_{++}(\R^{2,1})$, which is contractible due to the previous Proposition.
\end{proof}

\subsection{Period space}
Now we shall study the geometry of the period space of a hyperk\"ahler manifold $X$ itself, which is identified with the space $\Gr_{++}(H^2(X,\R))$ because of LeBrun's lemma. From now onwards the letter $V$ stands for a real vector space with a metric $q$ of signature $(3,n)$.

Let $U\subset V$ be a 3-dimensional space with $q|_U$ positive definite. Then any 2-plane in it is positive, and they constitute a rational curve $\bC P^1 = \Gr_{++}(U)\subset\Gr_{++}(V)$. In the case $V=H^2(X,\R)$ there exist a natural positive 3-subspace in $V$ spanned by the K\"ahler forms $\omega_I$, $\omega_J$ and $\omega_K$, and the corresponding family is the twistorial family. Because of that we shall call such curves {\it twistorial}, and denote the rational curve consisting of 2-planes in a positive 3-subspace $U$ by $\Gr_{++}(U)$. Twistorial curves are parametrized by the manifold $\Gr_{+++}(V)$ of positive 3-subspaces in the space $V$. Unlike the positive Grassmannian of 2-planes $\Gr_{++}(V)$, the Grassmannian of positive 3-subspaces $\Gr_{+++}(V)$ do not carry a natural complex structure (for example, in the case of period space of $K3$ surfaces its real dimension equals $57$). On the other hand, the rational curves in $\Gr_{++}(V)$ are parametrized by the Hilbert scheme, which is a scheme over $\bC$, so twistorial curves can be deformed to non-twistorial. It is also clear that the normal bundles of the twistorial curves are ample. One can map a curve of any genus into twistorial curve via ramified covering and then deform it into an embedded curve (that is possible because its normal bundle is generated by global sections due to a theorem of J.~Koll\'ar \cite{Ko}). That gives examples of many families of hyperk\"ahler manifolds over curves.

Let $v\in V$ be a positive vector. Then the space $v^\perp$ has signature $(2,n)$, and $\Gr_{++}(v^\perp)$ can be identified with a contractible bounded domain in a complex vector space. On the other hand, it is a divisor in the space $\Gr_{++}(V)$. We shall call it {\it a Cauchy divisor} (because of reasons explained in Section 4) and denote as $\Cau_v$. For any 2-plane $W\in\Gr_{++}(V)$ the set of Cauchy divisors passing through $W$ is the projectivization of the positive cone in $W^\perp$. It is easy to see that twistorial curves intersect the Cauchy divisors at one point.

\begin{fact}The period space contains no compact submanifolds of dimension greater than $1$.
\end{fact}
\begin{proof}If $X\subset\Per$ is a compact submanifold, then $\Cau_v\cap X$ is a divisor on $X$. If it is not empty, it is a compact submanifold of the Stein manifold $\Cau_v$, i.~e. a set of points. Therefore the dimension of $X$ needs to be equal to $1$.

One can prove this theorem the other way around: the period space $\Per$ is a subset of a quadric in a complex projective space, so it carries a positive $(1,1)$-form $\omega$. Due to the positivity one has $\int_X (\omega|_X)^{\dim X}>0$. However, the period space $\Per$ retracts onto the twistorial curve, so $\omega^{\dim X} = d\eta$ for some $(2\dim X-1)$-form $\eta$ and the integral needs to vanish unless $\dim X>1$.
\end{proof}

Now we can prove the main statement.

\begin{fact}Any family of hyperk\"ahler manifolds over a compact simply connected base has essential dimension not greater than 1.
\end{fact}
\begin{proof}The image of the period map $\per(B)$ is a compact submanifold in the period space $\Per$, thus a curve or a point.
\end{proof}

\subsection{Fibrations by complex tori}
In a way analogous to the stated above one can try to make a similar statement for the fibrations by $n$-dimensional complex tori. However, families of tori over a base of dimension greater that $1$ do exist, some of them were described in \cite{V2}. The Teichm\"uller space of the $n$-dimensional complex torus is the space of linear complex structures on the real vector space $V$ of dimension $2n$, hence isomorphic to $\GL^+(2n,\R)/\GL(n,\bC)$ (in particular, its real dimension is $4n^2-2n^2 = 2n^2$). It contains plenty of complex submanifolds~-- namely, for every Eucledian metric $g$ on $V$ one can define the submanifold consisting of complex structures orthogonal with respect to $g$. These submanifolds are the symmetric spaces $\SO(2n)/\rU(n)$, and the complex structure on these symmetric spaces coincides with the restriction of the complex structure on $\GL^+(2n,\R)/\GL(n,\bC)$. The real dimension of these submanifolds is equal to $\frac{2n(2n-1)}2 - n^2 = n^2 - n$.

On the other hand, any symplectic form $\psi$ on the vector space $V$ determines the submanifold $\Cau_\psi$ in the space of complex structures on it as follows: the subvariety $\Cau_\psi \subset \GL^+(2n,\R)/\GL(n,\bC)$ is the set of complex structures $I$ for which the form $\psi_I(u,v)=\psi(u,Iv)$ is symmetric and positive definite. A well-known 2-out-of-3 property of the unitary group asserts that a symplectic form and a Eucledian metric on a vector space uniquely determine the complex structure on it, in other words, the submanifold of complex structures orthogonal with respect to fixed metric intersects the submanifold $\Cau_\psi$ by $0$-cycle, thence $\dim_{\bC}\Cau_\omega=n^2 - \frac{n^2-n}2 = \frac{n^2+n}2$. Obviously, such submanifolds pass through each point in the Teichm\"uller space. But these submanifolds are biholomorphic to the Siegel upper-half space, in particular, they are Stein \cite[Ch.~2]{M}. Therefore the Teichm\"uller space of an $n$-dimensional complex torus cannot contain compact submanifolds of dimension greater than $\frac{n^2-n}2$.

So, we have proved the following
\begin{fact}Any family of $n$-dimensional complex tori over a compact simply connected base has essential dimension no greater than $\frac{n^2-n}2$.
\end{fact}

\subsection{Borcherds---Katzarkov---Pantev---Shepherd-Barron and Oguiso theorems}
For an integral vector $v\in H^2(X,\Z)$ the divisor $\Cau_v$ consists of complex structures $I$ for which one has $v\in H^{1,1}(X,I)$ (i.~e., the corresponding 2-plane in $H^2(X,\R)$ lies in the orthogonal to $v$). The subset $S_Z$ in the projectivization $\bP T^*{\Per}$ consisting of hyperplanes tangent to the integral Cauchy divisors is dense in an open subset of $\bP T^*{\Per}$ consisting of tangent hyperplanes to all Cauchy divisors $S$: hyperplanes tangent to Cauchy divisors are orthogonal to positive vectors, and lines spanned by positive integral vectors are dense in the positive cone.

Now, for any holomorphic disc $\Delta$ we can consider a point $p\in\Delta$ together with a Cauchy divisor $D$ passing through $p$. The tangent hyperplane $t=T_pD\subset T_p{\Per}$ is a point in $S$, and we can pick up a point $t'\in S_Z$ arbitrarily close to $t$. The divisor corresponding to $t'$ would be an integral Cauchy divisor intersecting $\Delta$. By shrinking $\Delta$, we can show that union of all integral Cauchy divisors interset any holomorphic disc in a dense subset. On the other hand, this union is the set where the Picard rank jumps. This proves Fact 1.4 (and thence the Fact 1.2).

\section{Appendix: Lorentzian K\"ahler metric on the period space}
In the present section, we shall also deal with the fibrations by hyperk\"ahler manifolds over noncompact or non-simply connected base. All fibrations are assumed to be such that the corresponding period map is an immersion (so that the universal covers of their bases are subvarieties in the period space $\Per$).

\begin{defn}Let $(X,g,I)$ be a complex manifold equipped with a metric of signature $(+,+,-,-,\dots)$. If the 2-form $\omega(u,v)=g(Iu,v)$ is closed, then we shall call such manifold a {\it Lorentzian K\"ahler manifold}.
\end{defn}

Due to the LeBrun's lemma stated in the Section 3, the period space $\Per$ is in fact the positive Grassmannian $\Gr_{++}(\R^{3,b_2-3})$. Its tangent space at the point $W\subset V$ is equal to $\Hom(W,W^\perp)$ and carries a natural (and hence $\SO(3,b_2-3)$-invariant) metric of indefinite signature $(2,2b_2-6)$. We shall denote in by $g_{\Per}$ and put $\omega_{\Per}(u,v)=g_{\Per}(Iu,v)$. This is a non-degenerate 2-form.

\begin{fact}$d\omega_{\Per}=0.$
\end{fact}
\begin{proof}Stabilizer of a point $W$ is a subgroup $\SO(W)\times\SO(W^\perp)\subset\SO(3,b_2-3)$. The group $\SO(W)=\SO(2)$ contains an operator $-\Id$, so the $\SO(3,b_2-3)$-invariant form $d\omega_{\Per}$ would be invariant under the fiberwise operator $-\Id$. However, it is a 3-form and thence vanishes.
\end{proof}

Nevertheless, this form is not exact because of the following reason:

\begin{fact}Let $U\subset V=\R^{3,b_2-3}$ be a 3-subspace with positive definite metric, and $\Gr_{++}(U)\subset\Gr_{++}(V)$ be the corresponding twistorial curve. Then the restriction $\omega_{\Per}|_{\Gr_{++}(U)}$ is the Fubini--Study form.
\end{fact}
\begin{proof}Actually, one can naturally associate such a form to any positive Grassmannian $\Gr_{++}$, and it would be compatible to its inclusions obtained from ones of vector spaces. The fact that this form on the $\bC P^1 = \Gr_{++}(\R^{3,0})$ is its Fubini--Study form may be regarded as its definition.
\end{proof}

As $\Gr_{++}(\R^{3,n})$ retracts onto any twistorial curve, its second cohomology is one-dimensional, and it is spanned by the cohomology class $[\omega_{\Per}]$. Moreover, one holds the following

\begin{fact}Any $\SO(3,b_2-3)$-invariant 2-form on the period space $\Per$ is a multiple of the form $\omega_{\Per}$.
\end{fact}
\begin{proof}Invariant tensors on homogenuous spaces are determined by its value at one point, which needs to be an invariant tensor in the tangent space under the action of the stabilizer of this point. So, one needs to check that the form $g_{\Per}|_W$ is a unique $\SO(2)\times\SO(1,b_2-3)$-invariant form on the space $\Hom(W,W^\perp)$. This 2-form defines a representation homomorphism $\Hom(W,W^\perp)\to\Hom(W^\perp,W)$. But $\Hom(W,W^\perp) = W \otimes (W^\perp)^*$. The space $W$ carries naturally the irreducible representation of $\SO(2)$, and $(W^\perp)^*$ is an irreducible representation of $\SO(1,b_2-3)$, so $\Hom(W,W^\perp)$ is an irreducible representation of $\SO(2)\times\SO(1,b_2-3)$. Schur's lemma implies that the space of homomorphisms between $\Hom(W,W^\perp)$ and its dual is one-dimensional. This proves the Proposition.
\end{proof}

Because of invariance of the form $\omega_{\Per}$ on the period space, its pullback $\per^*(\omega_{\Per})\in\Omega^2(\widetilde{B})$ is invariant under the $\pi(B)$-action, so it descends to the base and thus defines an invariant of a family $\fX\to B$ of hyperk\"ahler manifolds in the space of 2-forms on $B$. We shall denote the 2-form on $B$ obtained via this construction as $\omega_{\fX}\in\Omega^2(B)$.




One can find some similarities between the Lorentzian K\"ahler geometry of the positive Grassmannian $\Gr_{++}(\R^{3,n})$ and geometry of usual Lorentzian manifolds. For example, twistorial curves resemble timelike geodesics, whilst the divisors $\Cau_v$ are similar to Cauchy hypersurfaces. Nevertheless, this similarity is not complete: for example, while on a Lorentzian manifold $X$ containing a Cauchy hypersurface $M$ a bunch of timelike geodesics startled orthogonally from $M$ defines a decomposition $X=M\times\R$ \cite[Section 3]{ChN}, twistorial curves orthogonal to a Cauchy divisor $\Cau_v$ massively intersect, forming somewhat like a Lefschetz pencil.

\paragraph*{Acknowledgements.} I am grateful to M.~Verbitsky for turning my mind to the problem, fruitful discussions and help with preparation of the manuscript, and to D.~Kaledin and the anonymous referee of this paper on the M\"obius Contest for pointing to serious flaws.

This work was supported in part by the M\"obius Contest Foundation for Young Scientists.

{\sc \noindent Laboratory of Algebraic Geometry, \\
National Research University Higher School of Economics, \\
7 Vavilova Str., Moscow, Russia, 117312\\
e-mail: {\tt deevrod@mccme.ru}}

\end{document}